\documentclass[11pt]{amsart}%
\usepackage{amsfonts}
\usepackage{amsmath}
\usepackage{amssymb}
\usepackage{graphicx}%
\newtheorem{theorem}{Theorem}
\theoremstyle{plain}

\newtheorem{lemma}{Lemma}

\newtheorem{proposition}{Proposition}
\newtheorem{remark}{Remark}

\numberwithin{equation}{section}
\begin{document}
\title[Existence of nodal solutions...]{\textbf{Nodal solutions for elliptic equation involving the GJMS operators on
compact manifolds }}
\author{Mohamed Bekiri}
\address{Mohamed Bekiri, University Mustapha Stambouli; Mascara. Algeria}
\email{bekiri03@yahoo.fr}
\author{Mohammed Benalili }
\address{Mohammed Benalili Faculty Of Sciences, Mathematics Dept; University Abou-Bakr
Belka\^{\i}d; Tlemcen. Algeria}
\email{m\_benalili@mail.univ-tlemcen.dz}
\thanks{}
\date{December 1, 2017}
\dedicatory{ }
\begin{abstract}
In this paper we investigate the existence of nodal solutions to elliptic
problem involving the GJMS operators on Riemannian manifold with boundary.

\end{abstract}
\maketitle

\section{Introduction and preliminaries}

Let $\left(  M,g\right)  $ be an $n$-dimensional ( $n\geq3$) compact
Riemannian manifold, with boundary and let $k\geq1$ be an integer such that
$2k\leq n.$\newline In 1992 Graham-Jenne-Mason-Sparling \cite{GJMS} have
defined a familly of conformaly invariant differential operators ( GJMS
operators for short ). The construction of these operators is based on the
ambient metric of Fefferman-Graham (\cite{Fefferman-Graham1},
\cite{Fefferman-Graham2}).\newline More preciselly, for any Riemannian metric
$g$ on $M$, there exists a differential operator $P_{g}$, such that%

\[
P_{g}:=\Delta_{g}^{k}+lot
\]
where $\Delta_{g}:=-\operatorname{div}_{g}\left(  \nabla.\right)  $ is the
Laplace-Beltrami operator and $lot$ denotes differential terms of lower order.
The GJMS operator $P_{g}$ has the following nice properties.

\begin{enumerate}
\item $P_{g}$ is self-adjoint with respect to the $L^{2}$-scalar produit.

\item $P_{g}$ is conformally invariant in the sense that: if $\tilde
{g}=\varphi^{\frac{4}{n-2k}}g,$ is a metric conformal to $g$, for all $u\in
C^{\infty}\left(  M\right)  $, we have that%
\[
P_{g}\left(  u\varphi\right)  =\varphi^{\frac{n+2k}{n-2k}}P_{\tilde{g}}u.
\]
Setting $u\equiv1,$ we obtain
\[
P_{g}\varphi=\varphi^{\frac{n+2k}{n-2k}}P_{\tilde{g}}\left(  1\right)
\text{.}%
\]
To $P_{g}$ is attached a conformal invariant scalar function denoted $Q_{g}$,
called the $Q$-curvature. When $k=1$, the GJMS operator is the conformal
Laplacian and the $Q$-curvature is the scalar curvature (up to a constant).
When $k=2$, the GJMS operator is the Paneitz operator introduced in
\cite{Paneitz} and the corresponding $Q$-curvature was defined by
Branson-\O rsted \cite{Branson-Orsted} for four dimensional manifolds and then
was generalized to higher dimension manifolds by Branson (\cite{Branson1},
\cite{Branson2}). When $2k<n,$ the $Q$-curvature is given by $Q_{g}=\frac
{2}{n-2k}P_{g}\left(  1\right)  .$
\end{enumerate}

The divergence form of The GJMS operator was obtain by Robert.

\begin{proposition}
(\cite{Robert}) Let $P_{g}$ be the conformal $GJMS$ operator, then for any
$l\epsilon\left\{  0,1,...,k-1\right\}  $ there exists $A_{\left(  l\right)
}$ a smooth \ symmetric $\left(  2l,0\right)  $-tensor on $M$ suth that%
\begin{equation}
P_{g}=\Delta_{g}^{k}+\sum_{l=0}^{k-1}\left(  -1\right)  ^{l}\nabla
^{j_{l}...j_{1}}\left(  A_{\left(  l\right)  i_{1}...i_{l}j_{1}...j_{l}}%
\nabla^{i_{1}...i_{l}}\right)  \label{eq1}%
\end{equation}
where the indices are raised via the musical isomorphism.
\end{proposition}

Le $\left(  M,g\right)  $ be a smooth Riemannian compact manifold with
boundary of dimension $n.$ Let $k>1$ be an integer such that $n>2k$. The
Sobolev space $H_{k,0}^{2}\left(  M\right)  $ is defined as the completion of
$C_{c}^{\infty}\left(  M\right)  $ for the following norm (see Aubin \cite{T
Aubin})%
\[
\left\Vert u\right\Vert ^{2}:=\sum\limits_{i=0}^{k}\left\Vert \nabla
^{i}u\right\Vert _{2}%
\]
where $C_{c}^{\infty}\left(  M\right)  $ stands for the space of smooth
functions with compact supports in $M$.

This norm is equivalent to the following Hilbert norm (see \cite{Robert})%
\[
\left\Vert u\right\Vert _{H_{k}^{2}\left(  M\right)  }^{2}=\sum\limits_{i=0}%
^{k}\left\Vert \Delta^{\frac{i}{2}}u\right\Vert _{2}%
\]
where%
\[
\Delta^{\frac{i}{2}}u=\left\{
\begin{array}
[c]{ll}%
\Delta^{m}u & \text{\textit{if }}i=2m\text{ }\\
\nabla\Delta^{m}u & \text{\textit{if }}i=2m+1\text{ }%
\end{array}
\right.  \text{.}%
\]
We say that $P_{g}$ is coercive if there exists $\Lambda>0$, such that for all
$u\in H_{k.0}^{2}\left(  M\right)  $%
\[
\int\nolimits_{M}\left(  P_{g}u\right)  udv_{g}\geq\Lambda\left\Vert
u\right\Vert _{H_{k,0}^{2}\left(  M\right)  }^{2}\text{.}%
\]
We define the quantity $K_{0}\left(  n,k\right)  >0$
\[
\frac{1}{K_{0}\left(  n,k\right)  }:=\inf_{u\in C_{c}^{\infty}\left(
\mathbb{R}
^{n}\right)  -\left\{  0\right\}  }\frac{\int\limits_{%
\mathbb{R}
^{n}}\left(  \Delta^{\frac{k}{2}}u\right)  ^{2}dx}{\left(  \int\limits_{%
\mathbb{R}
^{n}}\left\vert u\right\vert ^{2^{\sharp}}dx\right)  ^{\frac{2}{2^{\sharp}}}}%
\]
as the best constant in the Euclidean Sobolev inequality $\left\Vert
u\right\Vert _{2^{\sharp}}^{2}\leq K\left\Vert \Delta^{\frac{k}{2}%
}u\right\Vert _{2}^{2}$ where $2^{\sharp}=\frac{2n}{n-2k}$ is the Sobolev
critical exponent.

We know from the work of Swanson (\cite{Swanson}) that,%
\[
\frac{1}{K_{0}\left(  n,k\right)  }=\pi^{k}\left(  \frac{\Gamma\left(
\frac{n}{2}\right)  }{\Gamma\left(  n\right)  }\right)  ^{\frac{2k}{n}}%
\prod\limits_{h=-k}^{k-1}\left(  n+2h\right)
\]
where $\Gamma$ denotes the Euler function. AMoreover $K_{0}\left(  n,k\right)
$ is attained by the radial function%
\begin{equation}
u_{0}\left(  x\right)  =\alpha_{n,k}^{\frac{n-2k}{4k}}\left(  1+\left\vert
x\right\vert ^{2}\right)  ^{-\frac{n-2k}{2}} \label{eq2}%
\end{equation}
with $\alpha_{n,k}=\prod\limits_{l=-k}^{k-1}\left(  n+2l\right)  .$

A Sobolev type inequality for the embeddings $H_{k,0}^{2}\left(  M\right)
\subset L^{2^{\sharp}}\left(  M\right)  $ proved by Mazumdar ( \cite{Mazumdar}%
) is given by the following lemma.

\begin{lemma}
\label{lem1} Let $\left(  M,g\right)  $ be a smooth $n$-dimensional compact
Riemannian manifold with boundary and let $k$ be a positive integer such that
$n>2k$. Then for any $\epsilon>0,$ there exists $B_{\epsilon}\in%
\mathbb{R}
$ such that for all $u\in H_{k,0}^{2}\left(  M\right)  $ one has
\begin{equation}
\left\Vert u\right\Vert _{2^{\sharp}}^{2}\leq\left(  K_{0}\left(  n,k\right)
+\epsilon\right)  \left\Vert \Delta_{g}^{\frac{k}{2}}u\right\Vert _{2}%
^{2}+B_{\epsilon}\left\Vert u\right\Vert _{H_{k-1}^{2}\left(  M\right)  }^{2}
\label{eq3}%
\end{equation}
where $K_{0}\left(  n,k\right)  $ is the best Euclidean constant..
\end{lemma}

In this paper we investigate the existence of a real positive number $\lambda$
and nodal solutions of the following polyharmonic Dirichlet problem%
\begin{equation}
\left\{
\begin{array}
[c]{ll}%
P_{g}u=\lambda f\left\vert u\right\vert ^{2^{\sharp}-2}u & \text{in }M\\
u=\phi_{1}\text{, }\frac{\partial u}{\partial\nu}=\phi_{2}\text{,..., and
}\frac{\partial^{k-1}u}{\partial\nu^{k-1}}=\phi_{k} & \text{on }\partial M
\end{array}
\right.  \label{eq4}%
\end{equation}
where $\left(  M,g\right)  $ is a compact Riemannian manifold of dimension
$n>2k$, with boundary $\partial M$, $P_{g}$ is the GJMS operator defined in
($\ref{eq1}$), $f\in C^{\infty}\left(  M\right)  $, $2^{\sharp}=\frac
{2n}{n-2k}$ denotes the critical Sobolev exponent,$\ \phi_{1},\phi
_{2},...,\phi_{k}\in C^{\infty}\left(  \partial M\right)  $ are boundary data
functions and $\nu$ is the unit outward normal vector field to $\partial M$.

When $\phi_{1}$ is of changing sign, $u$ is called a nodal solution of the
equation ($\ref{eq4}$).

The existence of solutions for polyharmonic operators with critical exponent
has been studied in many papers. We recall briefly some of them.

In \cite{Pucci-Serrin} Pucci-Serrin considered the following problem%
\[
\left\{
\begin{array}
[c]{l}%
\left(  -\Delta\right)  ^{k}u=\left\vert u\right\vert ^{k_{\ast}-2}u+\lambda
u\\
u\in H_{0}^{k}\left(  \Omega\right)  \text{, }\lambda>0
\end{array}
\right.
\]
where $k_{\ast}=\frac{2k}{n-2k}$ and $\Omega$ is a bounded starchaped domain
in $%
\mathbb{R}
^{n}$.

They proved that this problem has only the trivial solution $u\equiv0$ if
$\lambda<0$ and $\Omega$ is starshaped domain.

In \cite{H Grunau} Grunau has been intersted by the existence of positive
solutions to the semilinear polyharmonic Dirichlet%
\[
\left\{
\begin{array}
[c]{ll}%
\left(  -\Delta\right)  ^{k}u=\lambda u+u^{s-1} & \text{in }B\\
D^{\alpha}u=0,\text{ }\left\vert \alpha\right\vert \leq k-1 & \text{on
}\partial B
\end{array}
\right.
\]
where, $k\in%
\mathbb{N}
$, $B$ is the unit open ball in $%
\mathbb{R}
^{n}$, $n>2k$, $\lambda\in%
\mathbb{R}
$ and $s=\frac{2n}{n-2k}$ is the critical Sobolev exponent. In \cite{Yuxin
Ge1} Ge has obtained existence of positive\ solutions for semilinear
polyharmonic problem with Navier boundary conditions. In \cite{Yuxin Ge}
Ge-Wei-Zhou have proved the existence of solutions for the following
polyharmonic problem with perturbation%
\[
\left\{
\begin{array}
[c]{ll}%
\left(  -\Delta\right)  ^{k}u=\left\vert u\right\vert ^{s-2}u+f\left(
x,u\right)  & \text{in }\Omega\\
u=Du=...=D^{k-1}u=0 & \text{on }\partial\Omega
\end{array}
\right.
\]
where $\Omega$ is a smooth bounded domain in $%
\mathbb{R}
^{n}$, $s=\frac{2n}{n-2k}$ is the critical Sobolev exponent and $f\left(
x,u\right)  $ is a lower-order perturbation of $\left\vert u\right\vert
^{s-2}u$. In \cite{Saikat Mazumdar} \ Mazumdar obtained by Coron's topological
method, solutions to a nonlinear elliptic problem involving critical Sobolev
exponent for a polyharmonic operator on a Riemannian manifold. For more
details concerning higher order problems, we refer the reader to the general
monograph Gazzola-Grunau-Sweers \cite{Gazzola-Grunau-Sweers}.

In this paper, we extend the results obtained by the autors in \cite{Bekiri}
for equations containing Paneitz-Bran\c{c}on operator to the equations
involving the GJMS operator.

Our main results in the present work state as follows.

\begin{theorem}
\label{th1} Let $\left(  M,g\right)  $ be a compact Riemannian manifold of
dimension $n>2k$ with boundary $\partial M\neq0$. Let $A_{\left(  l\right)  }$
be a smooth symetric $\left(  2l,0\right)  $ tensor on $M$ for any
$l\in\left\{  0,1,...,k-1\right\}  $. Let $f\in C^{\infty}\left(  M\right)  $,
$f>0$ and $x_{0}$ a point in the interior of $M$ such that $f\left(
x_{0}\right)  =\max_{M}f$. We assume that the operator $P_{g}$ is coercive and
the following condition is satisfied
\[
\inf_{u\in H_{k}^{2}\left(  M\right)  -\left\{  0\right\}  }\frac{%
{\displaystyle\int_{M}}
\left(  \Delta_{g}^{\frac{k}{2}}u\right)  ^{2}dv_{g}+\sum\limits_{l=0}^{k-1}%
{\displaystyle\int_{M}}
A_{l}\left(  \nabla^{l}u,\nabla^{l}u\right)  dv_{g}}{\left(
{\displaystyle\int_{M}}
f\left\vert u\right\vert ^{2^{\sharp}}dv_{g}\right)  ^{\frac{2}{2^{\sharp}}}%
}<\frac{1}{\left(  f\left(  x_{0}\right)  \right)  ^{\frac{2}{2^{\sharp}}%
}K_{0}\left(  n,k\right)  }.
\]
\newline Then there exist a positive real number $\lambda$ and a non trivial
solution $u=w+h\in H_{k}^{2}\left(  M\right)  \cap C^{2k}\left(  M\right)  $
of the equation $\left(  \ref{eq4}\right)  $, where $w$ is a minimizer of the
functional $I$ defined on $H_{k,0}^{2}\left(  M\right)  $ by
\[
I\left(  w\right)  =\int_{M}\left(  \Delta_{g}^{\frac{k}{2}}w\right)
^{2}dv_{g}+\sum\limits_{l=0}^{k-1}\int_{M}A_{l}\left(  \nabla^{l}w,\nabla
^{l}w\right)  dv_{g}%
\]
under the constraint $\int_{M}f\left\vert w+h\right\vert ^{2^{\sharp}}%
dv_{g}=\gamma$, $h$ denotes the unique solution of the problem%
\[
\left\{
\begin{array}
[c]{ll}%
P_{g}h=0 & \text{in }M\\
h=\phi_{1}\text{, }\frac{\partial h}{\partial\nu}=\phi_{2},...\text{and }%
\frac{\partial^{k-1}h}{\partial\nu^{k-1}}=\phi_{k} & \text{on }\partial M
\end{array}
\right.
\]
and $\phi_{1}$, $\phi_{2},...,\phi_{k}$ are smooth functions on the boundary
$\partial M$ with $\phi_{1}$ of changing sign.
\end{theorem}

\begin{remark}
Since the function $\phi_{1}$ is of changing sign, the solutions obtained in
Theorem \ref{th1} are nodal.
\end{remark}

This paper is structured as follows. In section 2, we use the variational
method to construct a minimizing solution $w_{\gamma,q}$ to the subcritical
equations of type $\left(  \text{\ref{eq7}}\right)  $, while in section 3, we
prove that under some additional conditions the minimizing sequence $\left(
w_{\gamma,q}\right)  _{q}$ converges strongly to a non trivial solution of the
critical equation $\left(  \ref{eq6}\right)  $ when $q$ tends to the critical
Sobolev exponent $2^{\sharp}$. Finally, the last section is devoted to test
functions computations. These computations have their analogue in
\cite{Bekiri} when dealing with the Paneitz-Bran\c{c}on operator.

\section{Construction of Subcritical Solutions}

First, we extend uniquely the boundary data $\phi_{1},$ $\phi_{2},...,$
$\phi_{k}$ on the whole $M$.

\begin{lemma}
Let $\left(  M,g\right)  $ be a smooth Riemannian compact manifold with smooth
boundary and of dimension $n>2k.$ We assume that the operator $P_{g}$ defined
in $\left(  \ref{eq1}\right)  $ is coercive. Then there exists a unique $h\in
C^{2k,\alpha}\left(  M\right)  $, for some $\alpha\in\left(  0,1\right)
$,solution of the following problem
\begin{equation}
\left\{
\begin{array}
[c]{ll}%
P_{g}h=0 & \text{in }M\\
h=\phi_{1}\text{, }\frac{\partial h}{\partial\nu}=\phi_{2},...\text{and }%
\frac{\partial^{k-1}h}{\partial\nu^{k-1}}=\phi_{k} & \text{on }\partial M
\end{array}
\right.  \text{.} \label{eq5}%
\end{equation}

\end{lemma}

\begin{proof}
The proof is classical, the existence and the uniqueness follows from the
Lax-Milgram's theorem. Moreover the regularity of the solution follows from
general regularity theory.

Let $u=w+h$, the equation $\left(  \text{\ref{eq4}}\right)  $ will be written
as%
\begin{equation}
\left\{
\begin{array}
[c]{ll}%
P_{g}w=\lambda f\left\vert w+h\right\vert ^{2^{\sharp}-2}\left(  w+h\right)
& \text{in }M\\
w=\frac{\partial w}{\partial\nu}=...=\frac{\partial^{k-1}w}{\partial\nu^{k-1}%
}=0 & \text{on }\partial M
\end{array}
\right.  .\label{eq6}%
\end{equation}
Due to the lack of compactness of the embedding $H_{k,0}^{2}\left(  M\right)
\hookrightarrow L^{2^{\sharp}}\left(  M\right)  $, it is standard to use the
subcritical method. Given $q\in\left(  2,2^{\sharp}\right)  $, we consider the
subcritical problem%
\begin{equation}
\left\{
\begin{array}
[c]{ll}%
P_{g}w=\lambda f\left\vert w+h\right\vert ^{q-2}\left(  w+h\right)   &
\text{in }M\\
w=\frac{\partial w}{\partial\nu}=...=\frac{\partial^{k-1}w}{\partial\nu^{k-1}%
}=0 & \text{on }\partial M
\end{array}
\right.  \text{.}\label{eq7}%
\end{equation}
We define the functional $I$ on $H_{k,0}^{2}\left(  M\right)  $ by%
\[
I\left(  w\right)  =\int\nolimits_{M}wP_{g}wdv_{g}=\int_{M}\left(  \Delta
_{g}^{\frac{k}{2}}w\right)  ^{2}dv_{g}+\sum\limits_{l=0}^{k-1}\int_{M}%
A_{l}\left(  \nabla^{l}w,\nabla^{l}w\right)  dv_{g}\text{.}%
\]
Denote by%
\[%
\begin{array}
[c]{lll}%
\mu_{\gamma,q}:=\inf_{w\in\mathcal{H}_{q}}I\left(  w\right)   & \text{and} &
\mathcal{H}_{q}=\left\{
\begin{array}
[c]{cc}%
w\in H_{k,0}^{2}\left(  M\right)  \text{ such that} & \int\nolimits_{M}%
f\left\vert w+h\right\vert ^{q}dv_{g}=\gamma
\end{array}
\right\}
\end{array}
\]
where $\gamma$ is a constant such that
\begin{equation}
\int\nolimits_{M}f\left\vert h\right\vert ^{2^{\sharp}}dv_{g}<\gamma
\text{.}\label{eq8}%
\end{equation}

\end{proof}

Firstly, we show that the set $\mathcal{H}_{q}$ is not empty.

\begin{lemma}
Under the condition $\int\nolimits_{M}f\left\vert h\right\vert ^{2^{\sharp}%
}dv_{g}<\gamma$, the set $\mathcal{H}_{q}$ is not empty.
\end{lemma}

\begin{proof}
To do this, we set
\[
F\left(  t\right)  =\int\nolimits_{M}f\left\vert t\psi_{1}+h\right\vert
^{q}dv_{g}%
\]
where $\psi_{1}$ is the eigenfunction corresponding to the first eigenvalue
$\lambda_{1}$ of $\Delta_{g}^{k}$%
\[
\left\{
\begin{array}
[c]{ll}%
\Delta_{g}^{k}\psi_{1}=\lambda_{1}\psi_{1} & \text{in }M\\
\psi_{1}=\frac{\partial}{\partial\nu}\psi_{1}=...\frac{\partial^{k-1}%
}{\partial\nu^{k-1}}\psi_{1}=0 & \text{on }\partial M
\end{array}
\right.  \text{.}%
\]
The function $F$ is continuous, for $q$ close to $2^{\sharp}$, we have%
\[%
\begin{array}
[c]{lll}%
F\left(  0\right)  =\int\nolimits_{M}f\left\vert h\right\vert ^{q}%
dv_{g}<\gamma & \text{and} & \lim\limits_{t\rightarrow+\infty}F\left(
t\right)  =+\infty
\end{array}
\text{.}%
\]
So, there exists $t_{\gamma,q}>0$ such that
\[
F\left(  t_{\gamma,q}\right)  =\int\nolimits_{M}f\left\vert t_{\gamma,q}%
\psi_{1}+h\right\vert ^{q}dv_{g}=\gamma
\]
Consequently $t_{\gamma,q}\psi_{1}\in\mathcal{H}_{q}\neq\emptyset$.
\end{proof}

Now we prove that the minimum $\mu_{\gamma,q}$ is achieved by a smooth
function in $\mathcal{H}_{q}$.

\begin{proposition}
\label{prop3} Let $q\in\left(  2,2^{\sharp}\right)  $. If $P_{g}$ is coercive
and $\gamma>\int\nolimits_{M}f\left\vert h\right\vert ^{2^{\sharp}}dv_{g}$,
there exist a real number $\lambda_{\gamma,q}$ and a smooth function
$w_{\gamma,q}\in\mathcal{H}_{q}$ solution of the problem $\left(
\text{\ref{eq7}}\right)  $.
\end{proposition}

\begin{proof}
Since $P_{g}$ is coercive, there exists $C>0$ such that for all $u\in
H_{k,0}^{2}\left(  M\right)  $%
\begin{equation}
I\left(  u\right)  =\int\nolimits_{M}uP_{g}udv_{g}\geq C\left\Vert
u\right\Vert _{H_{k,0}^{2}} \label{eq9}%
\end{equation}
then $\mu_{\gamma,q}>0.$

Let $\left(  w_{i}\right)  _{i}$ be a minimizing sequence of the functional
$I$ on $\mathcal{H}_{q}$. By definition of $\mu_{\gamma,q}$ and from $\left(
\ref{eq9}\right)  $ for $i$ sufficiently large, we get%
\begin{equation}
\left\Vert w_{i}\right\Vert _{H_{k,0}^{2}\left(  M\right)  }\leq\frac{1}%
{C}\left(  \mu_{\gamma,q}+1\right)  \text{.} \label{eq10}%
\end{equation}

Hence, the sequence $\left(  w_{i}\right)  $ is bounded in $H_{k,0}^{2}\left(
M\right)  $. By reflexivity of this latter space, there exists a subsequence
of $\left(  w_{i}\right)  _{i}$ still labelled $\left(  w_{i}\right)  _{i}$
such that \newline%
\[%
\begin{array}
[c]{ll}%
\left(  a\right)  & w_{i}\rightharpoonup w_{\gamma,q}\text{ weakly in }%
H_{k,0}^{2}\left(  M\right)  \medskip\\
\left(  b\right)  & w_{i}\rightarrow w_{\gamma,q}\text{ strongly in }%
H_{k-1,0}^{2}\left(  M\right)  \text{ and in }L^{s}\left(  M\right)  \text{
for }s<2^{\sharp}\medskip\\
\left(  c\right)  & \left\Vert w_{\gamma,q}\right\Vert _{H_{k,0}^{2}}\leq
\lim\limits_{i}\inf\left\Vert w_{i}\right\Vert _{H_{k,0}^{2}}.
\end{array}
\]
\newline Consequently%
\[%
\begin{array}
[c]{ll}%
I\left(  w_{\gamma,q}\right)  & =%
{\displaystyle\int\nolimits_{M}}
\left(  \Delta^{\frac{k}{2}}w_{\gamma,q}\right)  ^{2}dv_{g}+\sum
\limits_{l=0}^{k-1}%
{\displaystyle\int_{M}}
A_{l}\left(  \nabla^{l}w_{\gamma,q},\nabla^{l}w_{\gamma,q}\right)
dv_{g}.\medskip\\
& \leq\lim_{i}\inf\left\Vert \Delta^{\frac{k}{2}}w_{i}\right\Vert _{2}%
^{2}+\lim_{i}\sum\limits_{l=0}^{k-1}%
{\displaystyle\int_{M}}
A_{l}\left(  \nabla^{l}w_{i},\nabla^{l}w_{i}\right)  dv_{g}.\medskip\\
& =\lim_{i}I\left(  w_{i}\right)  =\mu_{\gamma,q}.
\end{array}
\]
\newline Since%
\[
\int\nolimits_{M}f\left\vert w_{\gamma,q}+h\right\vert ^{q}dv_{g}%
=\lim\limits_{i}\int\nolimits_{M}f\left\vert w_{i}+h\right\vert ^{q}%
dv_{g}=\gamma
\]
we obtain%
\[
I\left(  w_{\gamma,q}\right)  =\mu_{\gamma,q}%
\]
If we write the Euler-Lagrange equation for $w_{\gamma,q}$, we find that
$w_{\gamma,q}$ is a weak solution of the equation%
\begin{equation}
\left\{
\begin{array}
[c]{ll}%
\Delta_{g}^{k}w_{\gamma,q}+\sum\limits_{l=0}^{k-1}\left(  -1\right)
^{l}\nabla^{j_{l}...j_{1}}\left(  A_{\left(  l\right)  i_{1}...i_{l}%
j_{1}...j_{l}}\nabla^{i_{1}...i_{l}}w_{\gamma,q}\right)  =\lambda_{\gamma
,q}f\left\vert w_{\gamma,q}+h\right\vert ^{q-2}\left(  w_{\gamma,q}+h\right)
& \text{in }M\\
w_{\gamma,q}=\frac{\partial}{\partial\nu}w_{\gamma,q}=...=\frac{\partial
^{k-1}}{\partial\nu^{k-1}}w_{\gamma,q}=0 & \text{on }\partial M
\end{array}
\right.  \text{.} \label{eq11}%
\end{equation}
where $\lambda_{\gamma,q}$ is the Lagrange multiplier

We easily get from standard boostrap arguments, that $w_{\gamma,q}\in
C^{2k,\alpha}\left(  M\right)  $ for some $\alpha\in\left(  0,1\right)  $.
\end{proof}

\section{Critical solutions and geometric conditions}

In this section, we will study the behavior of the minimizing sequence
$\left(  w_{\gamma,q}\right)  _{q}$ when $q$ tends to $2^{\sharp}$.

\begin{proposition}
Under the hypothesis $\int\nolimits_{M}f\left\vert h\right\vert ^{2^{\sharp}%
}dv_{g}<\gamma$, the sequence $\left(  w_{\gamma,q}\right)  _{q}$ is bounded
in $H_{k,0}^{2}\left(  M\right)  $. The Lagrange multiplier $\lambda
_{\gamma,q}$ is strictly positive and the sequence $(\lambda_{\gamma,q})_{q}$
is bounded when $q$ tends to $2^{\sharp}$.
\end{proposition}

\begin{proof}
The proof of this proposition is the same as in \cite{Bekiri}. Let us first
notice that the condition $\int\nolimits_{M}f\left\vert h\right\vert
^{2^{\sharp}}dv_{g}<\gamma$ remains valid for $q$ close to $2^{\sharp}.$ Now,
multiplying equation $\left(  \ref{eq11}\right)  $ by $w_{\gamma,q}$ and
integrating on $M$, we get
\[%
\begin{array}
[c]{ll}%
0\leq\mu_{\gamma,q} & =\int\nolimits_{M}w_{\gamma,q}P_{g}w_{\gamma,q}%
dv_{g}\medskip\\
& =\lambda_{\gamma,q}\int\nolimits_{M}f\left\vert w_{\gamma,q}+h\right\vert
^{q-2}\left(  w_{\gamma,q}+h\right)  w_{\gamma,q}dv_{g}\medskip\\
& =\lambda_{\gamma,q}\int\nolimits_{M}f\left\vert w_{\gamma,q}+h\right\vert
^{q-2}\left(  w_{\gamma,q}+h\right)  \left(  w_{\gamma,q}+h-h\right)
dv_{g}\medskip\\
& =\lambda_{\gamma,q}\left(  \gamma-\int\nolimits_{M}f\left\vert w_{\gamma
,q}+h\right\vert ^{q-2}\left(  w_{\gamma,q}+h\right)  hdv_{g}\right)  \text{.}%
\end{array}
\]
Thanks to H\"{o}lder's inequality, we have%
\[
\int\nolimits_{M}f\left\vert w_{\gamma,q}+h\right\vert ^{q-2}\left(
w_{\gamma,q}+h\right)  hdv_{g}\leq\gamma^{1-\frac{1}{q}}\left(  \int
\nolimits_{M}f\left\vert h\right\vert ^{q}\right)  ^{\frac{1}{q}}<\gamma
\]
and therefore, we obtain $\lambda_{\gamma,q}\geq0$.

Then, if $\lambda_{\gamma,q}=0$ we get $w_{\gamma,q}$ $=0$. A contradiction
with the fact that $w_{\gamma,q}\in\mathcal{H}_{q}$ and so $\int
\nolimits_{M}f\left\vert h\right\vert ^{q}dv_{g}<\gamma$.

We show that the sequence $\left(  w_{\gamma,q}\right)  _{q}$ is bounded in
$H_{k,0}^{2}\left(  M\right)  .$To do this, let $\psi_{1}$ be an eigenfunction
of $\Delta_{g}^{k}$ corresponding to the eigenvalue $\lambda_{1}$ such that%
\[
\left\{
\begin{array}
[c]{ll}%
\Delta_{g}^{k}\psi_{1}=\lambda_{1}\psi_{1} & \text{in }M\\
\psi_{1}=\frac{\partial}{\partial\nu}\psi_{1}=...\frac{\partial^{k-1}%
}{\partial\nu^{k-1}}\psi_{1}=0 & \text{on }\partial M\\
\int\nolimits_{M}\psi_{1}^{2}dv_{g}=1 &
\end{array}
\right.  .
\]
\newline Put%
\[
F\left(  t,q\right)  =\int\nolimits_{M}f\left\vert t\psi_{1}+h\right\vert
^{q}dv_{g}\text{.}%
\]
We have $t_{\gamma,q}\psi_{1}\in\mathcal{H}_{q}$ that is to say%
\[
F\left(  t_{\gamma,q},q\right)  =\int\nolimits_{M}f\left\vert t_{\gamma,q}%
\psi_{1}+h\right\vert ^{q}dv_{g}=\gamma.
\]
In the following we will show that $\frac{\partial F}{\partial t}\left(
t_{\gamma,q},q\right)  \neq0$. By contradiction, assume that $\frac{\partial
F}{\partial t}\left(  t_{\gamma,q},q\right)  =0.$

We remark that%
\[%
\begin{array}
[c]{ll}%
t_{\gamma,q}\frac{\partial F}{\partial t}\left(  t_{\gamma,q},q\right)  &
=qt_{\gamma,q}\int\nolimits_{M}f\left\vert t_{\gamma,q}\psi_{1}+h\right\vert
^{q-2}\left(  t_{\gamma,q}\psi_{1}+h\right)  \psi_{1}dv_{g}\medskip\\
& =q\int\nolimits_{M}f\left\vert t_{\gamma,q}\psi_{1}+h\right\vert
^{q-2}\left(  t_{\gamma,q}\psi_{1}+h\right)  \left(  t_{\gamma,q}\psi
_{1}+h-h\right)  dv_{g}\medskip\\
& =q\left(  \gamma-\int\nolimits_{M}f\left\vert t_{\gamma,q}\psi
_{1}+h\right\vert ^{q-2}\left(  t_{\gamma,q}\psi_{1}+h\right)  hdv_{g}\right)
=0
\end{array}
\]
hence
\begin{equation}
\gamma=\int\nolimits_{M}f\left\vert t_{\gamma,q}\psi_{1}+h\right\vert
^{q-2}\left(  t_{\gamma,q}\psi_{1}+h\right)  hdv_{g}\text{.} \label{eq12}%
\end{equation}
Now thanks to H\"{o}lder's inequality and the assumption $\int\nolimits_{M}%
f\left\vert h\right\vert ^{q}dv_{g}<\gamma,$ we infer that
\[
\int\nolimits_{M}f\left\vert t_{\gamma,q}\psi_{1}+h\right\vert ^{q-2}\left(
t_{\gamma,q}\psi_{1}+h\right)  hdv_{g}\leq\left(  \int\nolimits_{M}f\left\vert
t_{\gamma,q}\psi_{1}+h\right\vert ^{q}dv_{g}\right)  ^{1-\frac{1}{q}}\left(
\int\nolimits_{M}f\left\vert h\right\vert ^{q}dv_{g}\right)  ^{\frac{1}{q}%
}<\gamma\text{.}%
\]
which contradicts $\left(  \ref{eq12}\right)  $.

Since $\frac{\partial F}{\partial t}\left(  t_{\gamma,q},q\right)  \neq0$, the
implicit function theorem shows that $t_{\gamma,q}$ is a continuous function
of $q$. Consequently there exists a constant $C\left(  \gamma\right)  $
independent of $q$ such that%
\begin{equation}
\int\nolimits_{M}w_{\gamma,q}P_{g}w_{\gamma,q}dv_{g}\leq I\left(  t_{\gamma
,q}\psi_{1}\right)  =t_{\gamma,q}^{2}I\left(  \psi_{1}\right)  \leq C\left(
\gamma\right)  \text{.}\label{eq13}%
\end{equation}
From the coercivity of $P_{g}$, the sequence $\left(  w_{\gamma,q}\right)
_{q}$ is bounded in $H_{k,0}^{2}\left(  M\right)  $ when $q$ goes to
$2^{\sharp}$.

So, we can extract a subsequence of $\left(  w_{\gamma,q}\right)  _{q}$, still
denoted $\left(  w_{\gamma,q}\right)  _{q}$, such that%
\[%
\begin{array}
[c]{ll}%
\left(  a\right)  & w_{\gamma,q}\rightharpoonup w\text{ weakly in }H_{k,0}%
^{2}\left(  M\right)  \text{ as }q\rightarrow2^{\sharp}\medskip\\
\left(  b\right)  & w_{\gamma,q}\rightarrow w\text{ strongly in }H_{k-1}%
^{2}\left(  M\right)  \text{ and }L^{s}\left(  M\right)  \text{ for all
}s<2^{\sharp}\text{ as }q\rightarrow2^{\sharp}\medskip\\
\left(  c\right)  & w_{\gamma,q}\rightarrow w\text{ a.e in }M\text{ as
}q\rightarrow2^{\sharp}\text{.}%
\end{array}
\]
\newline Now, we will show that the sequence $\left(  \lambda_{\gamma
,q}\right)  _{q}$ is bounded for $q$ close to $2^{\sharp}$.

By the definition of $\lambda_{\gamma,q}$ and from the formula $\left(
\text{\ref{eq13}}\right)  $ and the fact that
\[
\int\nolimits_{M}f\left\vert w_{\gamma,q}+h\right\vert ^{q-2}\left(
w_{\gamma,q}+h\right)  hdv_{g}\leq\gamma^{1-\frac{1}{q}}\left(  \int
\nolimits_{M}f\left\vert h\right\vert ^{q}dv_{g}\right)  ^{\frac{1}{q}}<\gamma
\]
we obtain%
\[%
\begin{array}
[t]{ll}%
0<\lambda_{\gamma,q} & =\dfrac{\int\nolimits_{M}w_{\gamma,q}P_{g}w_{\gamma
,q}dv_{g}}{\int\nolimits_{M}f\left\vert w_{\gamma,q}+h\right\vert
^{q-2}\left(  w_{\gamma,q}+h\right)  w_{\gamma,q}dv_{g}}\medskip\\
& =\dfrac{\int\nolimits_{M}w_{\gamma,q}P_{g}w_{\gamma,q}dv_{g}}{\gamma
-\int\nolimits_{M}f\left\vert w_{\gamma,q}+h\right\vert ^{q-2}\left(
w_{\gamma,q}+h\right)  hdv_{g}}\medskip\\
& \leq\dfrac{I\left(  t_{\gamma,q}\psi_{1}\right)  }{\gamma-\gamma^{1-\frac
{1}{q}}\left(  \int\nolimits_{M}f\left\vert h\right\vert ^{q}dv_{g}\right)
^{\frac{1}{q}}}<C\left(  \gamma,h\right)  .
\end{array}
\]
Consequently, there exists a subsequence of $\left(  \lambda_{\gamma
,q}\right)  _{q},$still labelled $\left(  \lambda_{\gamma,q}\right)  _{q},$
which converges to $\lambda.$

Passing to the limit in equation (\ref{eq11}), we conclude that $w$ the weak
limit of the sequence ($w_{\gamma,q}$)$_{q}$ is a weak solution of the
critical equation (\ref{eq6}), so $u=w+h$ is a weak solution of the equation
(\ref{eq4}).

The regularity results obtained by Mazumdar \cite{Saikat Mazumdar} ( the proof
is based on ideas developed by Van der V\"{o}rst \cite{Van der vorst} and also
employed by Djadli-Hebey-Ledoux \cite{Djadli-Hebey-Ledoux} and Esposito-Robert
\cite{Esposito-Robert} for the case $k=2$) remain valid in our case and
consequently we conclude that our weak solution $w$ of the equation $\left(
\text{\ref{eq6}}\right)  $ is of class $C^{2k,\alpha}(M)$, for some $\alpha
\in\left(  0,1\right)  $.
\end{proof}

\begin{remark}
\label{rem1} If the boundary data $\left(  \phi_{1},\phi_{2},...,\phi
_{k-1}\right)  \not \equiv \left(  0,0,...,0\right)  $, then $u\not \equiv 0$
is a non trivial solution of the equation $\left(  \text{\ref{eq4}}\right)  $.
But if $\left(  \phi_{1},\phi_{2},...\phi_{k-1}\right)  =\left(
0,0,...,0\right)  $, then $h=0$ i.e. $u=w$. In this case, we will prove under
additional condition that $w$ is non trivial solution of the equation%
\begin{equation}
\left\{
\begin{array}
[c]{ll}%
P_{g}w=\lambda f\left\vert w\right\vert ^{2^{\sharp}-2}w & \text{in }M\\
w=\frac{\partial w}{\partial\nu}=...\frac{\partial^{k-1}w}{\partial\nu^{k-1}%
}=0 & \text{on }\partial M
\end{array}
\right.  .\label{eq14}%
\end{equation}
(\ref{eq14})if the condition $\left(  \text{\ref{eq15}}\right)  $ holds.
\end{remark}

\begin{proposition}
\label{prop5} Suppose that the minimizing sequence $\left(  w_{\gamma
,q}\right)  _{q}$ converges weakly to $w$. Put $\mu_{2^{\sharp}}%
=\lim\limits_{q\rightarrow2^{\sharp}}\mu_{\gamma,q}$. Assume that%
\begin{equation}
\frac{\mu_{2^{\sharp}}}{\gamma^{\frac{2}{2^{\sharp}}}}<\frac{1}{K_{0}\left(
f\left(  x_{0}\right)  \right)  ^{\frac{2}{2^{\sharp}}}} \label{eq15}%
\end{equation}
where $f(x_{0})=\max_{M}f(x)$, then $w$ is non trivial solution of the
equation $\left(  \text{\ref{eq14}}\right)  $.
\end{proposition}

\begin{proof}
We argue by contradiction. We may assume that $w\equiv0$, that is to say the
sequence $\left(  w_{\gamma,q}\right)  _{q}$ converges weakly to $0$ in
$H_{k,0}^{2}\left(  M\right)  $. The space $H_{k,0}^{2}\left(  M\right)  $ is
compactly embedded in $H_{k-1}^{2}\left(  M\right)  $. Therefore the sequence
$\left(  w_{\gamma,q}\right)  _{q}$ converges strongly to $0$ in $H_{k-1}%
^{2}\left(  M\right)  $ as $q\longrightarrow0$.

So, we have%
\[
\gamma^{\frac{2}{q}}=\left(  \int_{M}f\left\vert w_{\gamma,q}\right\vert
^{q}\right)  ^{\frac{2}{q}}\leq\left(  f(x_{0})\right)  ^{\frac{2}{q}}%
Vol_{g}\left(  M\right)  ^{\frac{2}{q}-\frac{2}{2^{\sharp}}}\left(  \int
_{M}\left\vert w_{\gamma,q}\right\vert ^{2^{\sharp}}\right)  ^{\frac
{2}{2^{\sharp}}}.
\]
We deduce from the Sobolev inequality $\left(  \text{\ref{eq3}}\right)  $ that
for any $\epsilon>0,$ there exists $B_{\epsilon}>0$ such that%
\[
\gamma^{\frac{2}{q}}\leq\left(  f(x_{0})\right)  ^{\frac{2}{q}}Vol_{g}\left(
M\right)  ^{\frac{2}{q}-\frac{2}{2^{\sharp}}}\left(  \left(  K_{0}%
+\epsilon\right)  \left\Vert \Delta^{\frac{k}{2}}w_{\gamma,q}\right\Vert
_{2}^{2}+B_{\epsilon}\left\Vert w_{\gamma,q}\right\Vert _{H_{k-1}^{2}}%
^{2}\right)  \medskip
\]
Since%
\[
\left\Vert \Delta^{\frac{k}{2}}w_{\gamma,q}\right\Vert _{2}^{2}=\mu_{\gamma
,q}-\sum\limits_{l=0}^{k-1}%
{\displaystyle\int_{M}}
A_{l}\left(  \nabla^{l}w_{\gamma,q},\nabla^{l}w_{\gamma,q}\right)  dv_{g}%
\]
it follows that%
\begin{equation}
\gamma^{\frac{2}{q}}\leq\left(  f(x_{0})\right)  ^{-\frac{2}{q}}Vol_{g}\left(
M\right)  ^{\frac{2}{2^{\sharp}}-\frac{2}{q}}\left\{
\begin{array}
[c]{c}%
\left(  K_{0}+\epsilon\right)  \left(  \mu_{\gamma,q}-\sum\limits_{l=0}^{k-1}%
{\displaystyle\int_{M}}
A_{l}\left(  \nabla^{l}w_{\gamma,q},\nabla^{l}w_{\gamma,q}\right)
dv_{g}\right) \\
+B_{\epsilon}\left\Vert w_{\gamma,q}\right\Vert _{H_{k-1}^{2}}^{2}%
\end{array}
\right\}  \text{.} \label{eq16}%
\end{equation}
Passing to the limit in (\ref{eq16}), choosing $\epsilon$ sufficiently small
and the fact that%

\[%
\begin{array}
[c]{l}%
\left\Vert w_{\gamma,q}\right\Vert _{H_{k-1}^{2}}^{2}=o\left(  1\right)
\medskip\\
Vol_{g}\left(  M\right)  ^{\frac{2}{q}-\frac{2}{2^{\sharp}}}=1+o\left(
1\right)  \medskip\\
\left\vert \sum\limits_{l=0}^{k-1}%
{\displaystyle\int_{M}}
A_{l}\left(  \nabla^{l}w_{\gamma,q},\nabla^{l}w_{\gamma,q}\right)
dv_{g}\right\vert \leq C\left\Vert w_{\gamma,q}\right\Vert _{H_{k-1}^{2}}%
^{2}=o\left(  1\right)  \medskip
\end{array}
\]

We infer that%
\[
\frac{\mu_{2^{\sharp}}}{\gamma^{\frac{2}{2^{\sharp}}}}>\frac{1}{\left(
f\left(  x_{0}\right)  \right)  ^{\frac{2}{2^{\sharp}}}K_{0}}\text{.}%
\]

A contradiction with the condition $\left(  \ref{eq15}\right)  $.

Hence $w\not \equiv 0$.
\end{proof}

\section{Estimates and test functions}

The purpose of this paragraph is to prove that the geometric condition
obtained in $\left(  \text{\ref{eq15}}\right)  $ is verified. The natural
strategy is to evaluate the quotient $Q=\frac{\mu_{2^{\sharp}}}{\gamma
^{\frac{2}{2^{\sharp}}}}$ at some good test functions where
\[%
\begin{array}
[c]{lll}%
\mu_{2^{\sharp}}=\lim_{q\rightarrow2^{\sharp}}\mu_{\gamma,q} & \text{and} &
\gamma=\int\nolimits_{M}f\left\vert u\right\vert ^{2^{\sharp}}dv_{g}%
\end{array}
.
\]

Let $\delta\in\left(  0,\frac{i_{g}\left(  M\right)  }{2}\right)  ,$ where
$i_{g}\left(  M\right)  $ is the injectivity radius and $x_{0}\in M$ such that
$f\left(  x_{0}\right)  =\max_{M}f.$ We let also $\eta\in C^{\infty}\left(
M\right)  $ be a cut-off function such that $\eta\equiv1$ in $B_{g}\left(
x_{0},\delta\right)  $ $\eta\equiv0$ in $M-B_{g}\left(  x_{0},2\delta\right)
$, where $B_{g}\left(  x_{0},\delta\right)  $ denotes the geodesic ball of
center $x_{0}$ and radius $\delta$. We define the following radial smooth
function%
\[
u_{\epsilon}\left(  x\right)  =\eta\left(  x\right)  \left(  \frac{\epsilon
}{\epsilon^{2}+r^{2}}\right)  ^{\frac{n-2k}{2}}=\eta\left(  x\right)
\epsilon^{-\frac{n-2k}{2}}u_{0}\left(  \frac{\exp_{x_{0}}^{-1}\left(
x\right)  }{\epsilon}\right)
\]
where $\exp_{x_{0}}$ is the exponential map at $x_{0}$ and $u_{0}$ is the
extremal function for the best Euclidean Sobolev inequality given by $\left(
\ref{eq2}\right)  $ and $r=d_{g}\left(  x_{0},x\right)  $ denotes the geodesic
distance to the point $x_{0}.$

We will now compute the estimates of
\[%
\begin{array}
[c]{lll}%
\mu_{2^{\sharp}}\left(  u_{\epsilon}\right)   & = & \int\limits_{M}\left(
\Delta_{g}^{\frac{k}{2}}u_{\epsilon}\right)  ^{2}dv_{g}+\sum\limits_{l=0}%
^{k-1}%
{\displaystyle\int_{M}}
A_{l}\left(  \nabla^{l}u_{\epsilon},\nabla^{l}u_{\epsilon}\right)  dv_{g}\\
\gamma\left(  u_{\epsilon}\right)   & = & \int\limits_{M}f\left\vert
u_{\epsilon}\right\vert ^{2^{\sharp}}dv_{g}%
\end{array}
\text{ and}%
\]
for $\epsilon$ sufficiently small.

With similar computations to \cite{Saikat Mazumdar} in the proof of
Proposition $5.1$, we get that%
\begin{equation}
\lim_{\epsilon\rightarrow0}\int\limits_{M}\left(  \Delta_{g}^{\frac{k}{2}%
}u_{\epsilon}\right)  ^{2}dv_{g}=\int\limits_{%
\mathbb{R}
^{n}}\left(  \Delta^{\frac{k}{2}}u_{0}\right)  ^{2}dx \label{eq17}%
\end{equation}
and
\begin{equation}
\lim_{\epsilon\rightarrow0}\left(  \sum\limits_{l=0}^{k-1}%
{\displaystyle\int_{M}}
A_{l}\left(  \nabla^{l}u_{\epsilon},\nabla^{l}u_{\epsilon}\right)
dv_{g}\right)  =0 \label{eq18}%
\end{equation}
and
\begin{equation}
\lim_{\epsilon\rightarrow0}\int\limits_{M}f\left\vert u_{\epsilon}\right\vert
^{2^{\sharp}}dv_{g}=f\left(  x_{0}\right)  \int\limits_{%
\mathbb{R}
^{n}}\left\vert u_{0}\right\vert ^{2^{\sharp}}dx \label{eq19}%
\end{equation}
where $u_{0}$ is the extremal function for the best Euclidean Sobolev
inequality given by $\left(  \text{\ref{eq2}}\right)  .$

Thanks to (\ref{eq17}), (\ref{eq18}) and (\ref{eq19}), we get that%
\[
\lim_{\epsilon\rightarrow0}Q\left(  u_{\epsilon}\right)  =\lim_{\epsilon
\rightarrow0}\frac{\mu_{2^{\sharp}}\left(  u_{\epsilon}\right)  }{\left(
\gamma\left(  u_{\epsilon}\right)  \right)  ^{\frac{2}{2^{\sharp}}}}=\frac
{1}{\left(  f\left(  x_{0}\right)  \right)  ^{\frac{2}{2^{\sharp}}}%
K_{0}\left(  n,k\right)  }%
\]
where%
\[
\frac{1}{K_{0}\left(  n,k\right)  }=\frac{\int\limits_{%
\mathbb{R}
^{n}}\left(  \Delta^{\frac{k}{2}}u_{0}\right)  ^{2}dx}{\left(  \int\limits_{%
\mathbb{R}
^{n}}\left\vert u_{0}\right\vert ^{2^{\sharp}}dx\right)  ^{\frac{2}{2^{\sharp
}}}}%
\]

\end{document}